\documentclass[12pt]{article}

\usepackage[utf8]{inputenc}
\usepackage[alphabetic]{amsrefs}
\usepackage{amssymb}
\usepackage{mathrsfs}
\usepackage{bbm}
\usepackage{xcolor}
\usepackage{lscape}
\usepackage{subfig}
\usepackage{titlefoot}
\usepackage{float}
\usepackage{fullpage}
\usepackage{stmaryrd}
\usepackage{amsthm}
\usepackage{enumitem}
\usepackage{amsmath}
\usepackage{thmtools}
\usepackage{ifthen}

\usepackage{amsthm}

\usepackage[colorlinks=true, linkcolor=magenta, citecolor=violet]{hyperref}

\usepackage[export]{adjustbox}

\theoremstyle{plain}
\newtheorem{theorem}{Theorem}[section]
\newtheorem{lemma}[theorem]{Lemma}
\newtheorem{proposition}[theorem]{Proposition}

\newtheorem{corollary}[theorem]{Corollary}

\theoremstyle{definition}
\newtheorem{definition}[theorem]{Definition}
\newtheorem{example}[theorem]{Example}
\newtheorem{question}[theorem]{Question}
\newtheorem*{question*}{Question}

\theoremstyle{remark}
\newtheorem{remark}[theorem]{Remark}

\newtheorem*{To show}{To show}

\let\phi\varphi
\let\epsilon\varepsilon
\renewcommand\emptyset\varnothing
\DeclareSymbolFont{extraup}{U}{zavm}{m}{n}
\DeclareMathSymbol{\varheartsuit}{\mathalpha}{extraup}{86}
\DeclareMathSymbol{\vardiamondsuit}{\mathalpha}{extraup}{87}

\newcommand\Aut{\operatorname{Aut}}

\newcommand{\Stab}{\operatorname{Stab}}

\newcommand{\Sub}{\operatorname{Sub}}
\newcommand{\Lie}{\operatorname{Lie}}
\newcommand{\Hit}{\operatorname{Hit}}
\newcommand{\Miss}{\operatorname{Miss}}

\newcommand{\diam}{\operatorname{diam}}
\newcommand{\supp}{\operatorname{supp}}

\newcommand{\Ad}{\mathrm{Ad}}
\newcommand{\GL}{\mathrm{GL}}

\newcommand{\lev}{\mathrm{lev}}

\newcommand{\ncm}{\overline{\langle \mu \rangle}}

\newcommand\bbK{\mathbb{K}}

\newcommand\bbN{\mathbb{N}}

\newcommand\bbR{\mathbb{R}}

\newcommand\bbZ{\mathbb{Z}}

\newcommand\fraka{\mathfrak{a}}

\newcommand\frakg{\mathfrak{g}}

\newcommand\frakl{\mathfrak{l}}

\newcommand\frakn{\mathfrak{n}}

\newcommand\calC{\mathcal{C}}

\newcommand\calE{\mathcal{E}}

\newcommand\calH{\mathcal{H}}

\newcommand\calK{\mathcal{K}}

\newcommand\calM{\mathcal{M}}

\newcommand\calW{\mathcal{W}}

\makeatletter
\renewcommand{\@fnsymbol}[1]{
  \ensuremath{
    \ifcase#1\or
      \spadesuit\or
      \varheartsuit\or
      \vardiamondsuit\or
      \clubsuit\or
      \spadesuit\spadesuit\or
      \heartsuit\heartsuit\or
    \fi
  }
}
\makeatother
\title{Compact Invariant Random Subgroups}
\author{Tal Cohen\footnote{Weizmann Institute of Science.}, Helge Gl\"{o}ckner\footnote{Paderborn University, Institute of Mathematics.}, Gil Goffer\footnote{University of California at San Diego.}, and Waltraud Lederle\footnote{Bielefeld University.}}

\begin{document}

\maketitle

\begin{abstract} We study ergodic invariant random subgroups that give full measure to the subset of compact subgroups. We show that in real Lie groups, compactly generated $p$-adic Lie groups, locally compact hyperbolic groups and infinitely ended groups they are always contained in a compact normal subgroup. In general $p$-adic Lie groups, we show they are contained in the locally elliptic radical.
In totally disconnected locally compact groups, we show they are contained in the intersection of all Levi subgroups of inner automorphisms.
\end{abstract}

\section{Introduction}
An \emph{invariant random subgroup} (IRS) of a locally compact group $G$ is a conjugation-invariant Borel probability measure on the Chabauty space $\Sub(G)$ of closed subgroups of $G$. Introduced by Abert--Glasner--Vir{\'a}g \cite{AGV}, IRSs simultaneously generalise normal subgroups and lattices: a normal subgroup $N\trianglelefteq G$ gives rise to the Dirac mass $\delta_N$, while a lattice $\Gamma\le G$ gives rise to the push-forward of the Haar measure on $G/\Gamma$ under the stabiliser map. Since their introduction, invariant random subgroups have become a central object of study in measured group theory, with deep connections to rigidity theory, ergodic theory, and the structure of locally compact groups; see, e.g., \cite{AGV, bader2016amenable, PropertyM, stuck1994stabilizers}.
IRSs also arise naturally from probability measure preserving (pmp) actions. Indeed, given a pmp action of $G$ on $(X,\mu)$, the stabilizer of a $\mu$-random point defines an IRS. In fact, every IRS arises in this way; see \cite[Theorem 2.6]{7s}.

Given an IRS $\mu$ of $G$, we consider the closed, normal subgroup $ \ncm := \overline{\langle\bigcup\supp(\mu)\rangle}\le G$, often called \emph{the normal closure of $\mu$}. It is the smallest closed subgroup $N\le G$ such that $\mu(\mathrm{Sub}(N))=1$. 

In many contexts, classification results for IRSs mirror, and extend, classical rigidity theorems for lattices and normal subgroups, see \cite{stuck1994stabilizers,bader2016amenable,bowen2015invariant} and more. An inspiring demonstration of this is Bader--Duchesne--L{\'e}cureux's result showing that for an IRS $\mu$ that almost surely chooses amenable subgroups, $\ncm$ is amenable.

In the current paper, we investigate \emph{compact IRSs}, i.e., IRSs that almost surely choose compact subgroups. Our paper is guided by the following (informal) question:
\begin{question*}
    If $\mu$ is a compact IRS of $G$, how close is $\ncm$ to being compact?
\end{question*}

We assume henceforth that $G$ is locally compact second-countable (\emph{lcsc} for short). In several cases, such as real Lie groups, compactly generated $p$-adic Lie groups, and locally compact hyperbolic groups, we are able to provide the strongest possible result, showing that $\ncm$ is compact; see Theorems \ref{thm:real}, \ref{thm:p-adic}, and \ref{thm:hyp or ends} below. 

More generally, we consider the following topologically characteristic subgroups of~$G$; see \cite[Section 2]{Cornulier_commability} for a nice overview.
The \emph{polycompact radical} of $G$, denoted by $\calW(G)$, is the union of all compact normal subgroups of $G$. It is not necessarily closed, but it is closed whenever~$G$ is compactly generated
(see \cite{Cornulier_commability})
or a $p$-adic Lie group (Remark~\ref{closedpadic}).
The \emph{locally elliptic radical} of $G$, denoted by $\calE(G)$, is the maximal locally elliptic\footnote{A subgroup is \emph{locally elliptic} if every compact subset is contained in a compact subgroup.}, normal subgroup of $G$; it is always closed and contains $\overline{\calW(G)}$.

To measure `how close' $\ncm$ is to being compact, we distinguish three different possible conclusions, ranked from the strongest to the weakest:

\begin{enumerate}[label=(\arabic*)]
        \item\label{strong} If $\mu$ is a compact ergodic IRS of $G$, then $\ncm $ is compact.
        \item\label{middle}
        If $\mu$ is a compact IRS of $G$, then $\ncm \le\overline{\calW(G)}$.
        \item\label{weak}
        If $\mu$ is a compact IRS of $G$, then $\ncm \le \calE(G)$.
\end{enumerate}
Since $\overline{\calW(G)} \le \calE(G)$, Conclusion \ref{middle} implies Conclusion \ref{weak}. The fact that Conclusion \ref{strong} implies Conclusion \ref{middle} follows from the ergodic decomposition theorem, see Lemma \ref{lem:non-ergodic}. A similar proof shows that it
does not
matter if one assumes $\mu$ is ergodic in Conclusions \ref{middle} and \ref{weak}. 
For Conclusion \ref{strong}, however, it is important to assume ergodicity, in order to avoid immediate counterexamples such as choosing larger and larger compact normal subgroups with decreasing probabilities. 
Moreover, if $\calW(G)$ is closed, then Conclusion \ref{middle} and Conclusion \ref{strong} are equivalent, as is explained in the second part of Lemma \ref{lem:non-ergodic}.

As a first result, we show that for real Lie groups the strongest conclusion holds; note that we do not assume anything about connectedness, but we do assume that real and $p$-adic Lie groups are second countable.
\begin{theorem}\label{thm:real}
    Let $G$ be a real Lie group, and $\mu$ a compact ergodic IRS of $G$. 
    Then $\ncm$ is compact.
\end{theorem}
As a corollary, we
deduce that this holds as well for any lcsc group $G$ such that $G/G^0$ is either compact or countable, where $G^0$ denotes the connected component of the identity; see Corollary \ref{Cor:AlmostConnected}.
For $p$-adic Lie groups, we show the following.
\begin{theorem}\label{thm:p-adic}
    Let $G$ be a $p$-adic Lie group.
    \begin{enumerate}[label=(\roman*)]
        \item\label{general-p-adic} If $\mu$ is a compact IRS of $G$, then $\ncm \le \calE(G)$.
        \item\label{cg-or-alg} Assume $G$ is compactly generated or algebraic. If $\mu$ is a compact ergodic IRS of $G$, then $\ncm$ is compact.
    \end{enumerate}
\end{theorem}

A key ingredient in the proof of the compactly generated case is a structural result of independent interest (Theorem~\ref{thm:Gloeckner W(G) open in lev(G)}): in a compactly generated $p$-adic Lie group, $\calW(G)$ is open in $\lev(G)$, the intersection of all Levi subgroups of inner automorphisms.

Using actions on boundaries of hyperbolic spaces, we further show the following.

\begin{theorem}\label{thm:hyp or ends}
    Let $G$ be a compactly generated lcsc group. Assume that $G$ is hyperbolic or has infinitely many ends.
    If $\mu$ is a compact IRS of $G$, then $\ncm$ is contained in $\calW(G)$, which is compact.
\end{theorem}

The following is an example of a compact IRS $\mu$ of an lcsc group $G$, for which $\ncm=G$ is non-compact. 
\begin{example}\label{exa:middle not strong G_r}
    Consider the group $G_{\eta}$ of automorphisms of a regular tree preserving a horosphere $\calH$ corresponding to a point $\eta$ in the tree boundary. Equivalently, $G_{\eta}$ consists of all tree automorphisms that fix pointwise some ray representing $\eta$. This group is locally elliptic, but not compact; its polycompact radical is a proper, dense subgroup.
    Let $\mu$ be the compact IRS defined as the pointwise stabiliser of a random subset of vertices of $\calH$ (chosen according to a Bernoulli product measure)\footnote{This IRS is essentially the intersectional IRS of $G_\eta$ associated with a vertex stabiliser $K$, as defined in \cite[Section 3]{Intersectional}. Although they only consider countable groups, their definition works as long as the normaliser of $K$ is open.}. Since $\mu$ is the pushforward of a Bernoulli measure on $2^\calH$, and the Bernoulli measure is ergodic, $
    \mu$ is an ergodic IRS. Note that, for every finite subset $A \subset \calH$, the probability of not fixing any element of $A$ is positive. It follows that $\ncm=G_\eta$.
    Since $G_\eta$ is not compact and $G_\eta=\overline{\calW(G_\eta)}$, we get that $G_\eta$ satisfies Conclusion \ref{middle}, but not Conclusion \ref{strong}.
\end{example}

Observe that the group $G_\eta$ from Example \ref{exa:middle not strong G_r} is not compactly generated. We do not know the answer to the following:
\begin{question}
    If $\mu$ is an ergodic compact IRS of a compactly generated lcsc group $G$, is $\ncm$ necessarily compact?
\end{question}

We further ask:
\begin{question}
    If $\mu$ is a compact IRS of an lcsc group $G$, is $\ncm$ necessarily contained in $\overline{\calW(G)}$? Is it necessarily contained in $\calE(G)$?
\end{question}

\begin{remark}
    If $G$ satisfies Conclusions \ref{strong}, \ref{middle}, or \ref{weak} above, $\Gamma\le G$ is a lattice, and $N\trianglelefteq \Gamma$ is a finite normal subgroup, then $N$ is contained in some compact normal subgroup of $G$, in $\overline{\calW(G)}$, or in $\calE(G)$, respectively. In general, it seems to be open whether a finite, normal subgroup of a nonuniform lattice is necessarily contained in the polycompact radical of the ambient group.
\end{remark}

Given a compact IRS $\mu$ of $G$, we can define a conjugation-invariant probability measure~$\nu$ on $G$ that is supported on elliptic elements by sampling a $\mu$-random subgroup $H$ and then a Haar-random element $h\in H$ (see Proposition \ref{AveragingMu}).
Conversely, given such a measure $\nu$ on $G$, we get a compact IRS of $G$ via the map $g \mapsto \overline{\langle g \rangle}$.
Thus, we can go back and forth between compact IRSs and conjugation-invariant probability measures supported on elliptic elements.
Note that it is not clear whether $\supp(\nu)$ is compact even when $\nu$ is supported on the conjugacy class of a single element, see for example \cite{Raja2025}, where this is proved for (M)-by-abelian groups.

By studying $\supp(\nu)$, via combining a Poincar\'e recurrence argument with Willis' structure theory of totally disconnected, locally compact groups, we show:

\begin{theorem}
    Let $G$ be a totally disconnected, lcsc group and $\mu$ a compact IRS of $G$. Then $\ncm$ is contained in $\bigcap_{g \in G} \lev(g)$, the intersection of all Levi subgroups of inner automorphisms\footnote{This subgroup was introduced as the \emph{approximate center} by Willis in \cite{Willis2023}.}.
\end{theorem}

\paragraph{Acknowledgements. }
We are indebted to Pierre-Emmanuel Caprace for proposing this problem and for many helpful discussions about it.
We thank Romain Tessera and Tianyi Zheng for valuable discussions and insightful comments.
The first author is co-funded by the European Union (ERC, Function Fields, 101161909).
The second author thanks the Deutsche Forschungsgemeinschaft (DFG, German Research Foundation) for support via the CRC TRR 358 – Project-ID 491392403.
Part of this work was completed when the fourth author was visiting UCSD, she thanks Tianyi Zheng for her hospitality. Her visit was financially partly supported by NSF DMS 2348143. Also, she was a F.R.S.-FNRS postdoctoral researcher.
\section{Preliminaries}
Given a locally compact group $G$, we denote by $\Sub(G)$ the space of closed subgroups of $G$, endowed with the \emph{Chabauty topology}. A subbasis for this topology is given by:
\begin{align*}
    \Hit(V) := \{H \leq G \mid H \cap V \neq \emptyset \} &\qquad \quad V \subset G \text{ open} \\
    \Miss(K) := \{H \leq G \mid H \cap K = \emptyset \} &\qquad \quad K \subset G \text{ compact}.
\end{align*}
This space is compact and Hausdorff. If $G$ is second countable, then it is also metrizable. For an introduction, see \cite{Gelander2018}.

\begin{lemma}
\label{lem:non-ergodic}
Let $G$ be an lcsc group. 
If every ergodic compact IRS is contained in a compact normal subgroup, then every compact IRS is contained in $\overline{\calW(G)}$.
The converse holds if $\calW(G)$ is closed.
\end{lemma}
\begin{proof}
The first part is immediate from the ergodic decomposition theorem \cite[Theorem 3.22]{Glasner2003}. 
To see that the converse holds if $\calW(G)$ is closed, observe that, by \cite[(1) in the proof of Thm.\ 2.8]{Cornulier_commability}, $\calW(G)$ can be written as an ascending union $\calW(G) = \bigcup_{n \in \bbN} O_n$ of compact relatively open subgroups $O_n\triangleleft G$. It follows that $\{K \le \calW(G) \mid K \text{ compact}\} = \bigcup \Sub(O_n)$, so, by ergodicity, there exists $n$ such that $\mu(\Sub(O_n))=1$.
\end{proof}
If $G$ is an lcsc group, we say that a subgroup $H\le G$ is \emph{cofinite} if there is a Borel $G$-invariant probability measure on $G/H$.
\begin{lemma}
\label{lem:supp-confinite}
Let 
$H\le G$ be a subgroup. Then the normaliser of $H$ is cofinite 
if and only if there is an IRS that is supported on its conjugacy class; that is, an IRS $\mu$ such that 
\[
\mu\left(\left\{ gHg^{-1}\middle|g\in G\right\} \right)=1.
\]
\end{lemma}
\begin{proof}
Denote by $N=N_G(H)$ the normalizer of $H$ in $G$. The map $gN\mapsto gHg^{-1}$ serves as an isomorphism of measurable $G$-spaces between the quotient $G/N$ (with $G$ acting by left multiplication) and the subset $\{gHg^{-1} ~|~ g\in G\}\subset \Sub(G)$ (with $G$ acting by conjugation). 
Therefore, the existence of a $G$-invariant probability measure on $G/N$ is equivalent to the existence of an IRS of $G$, supported on $\{gHg^{-1} ~|~ g\in G\}$.
\end{proof}

Let $\calM(G)$ denote the space of Borel probability measures on $G$. For a compact subgroup $H\le G$, let $m_H\in \calM(G)$ be the normalised Haar probability measure on $H$, viewed as a measure on $G$ supported on $H$.

\begin{lemma}\label{Borel}
    Let $\mathcal{C}\subseteq\mathrm{Sub}(G)$ be the set of compact subgroups. Then $\calC$ is measurable and
     the map  $\calC \to \calM(G), H\mapsto m_{H}$ is Borel measurable.
\end{lemma}
\begin{proof}
To show that $\calC$ is measurable, recall that $G$ is $\sigma$-compact and write $G = \bigcup_{n} O_n$, where $O_0 \subset O_1 \subset \dots$ and $O_n \subset G$ is open and has compact closure. Write $K_n := \overline{O_n}$.
Then $\calC = \bigcup_{n} \{H \in \Sub(G) \mid H \subset K_n\}$ is a countable union of closed subsets.

Fix a continuous nonnegative compactly supported function $f:G\to \bbR$
with $f(1)=1$. For each closed subgroup $H\le G$, let $m_{f}(H)\in\mathcal{M}(G)$
denote the unique left Haar measure on $H$ normalized by $\int_{G}f\,dm_{f}(H)=1.$
By \cite[Claim A.2]{MR3624401}, the map $H\mapsto m_{f}(H)$ from $\mathrm{Sub}(G)$
to $\mathcal{M}(G)$ is continuous. 

For each $n$, the map $H\mapsto m_{f}(H)(K_{n})$ (from $\mathrm{Sub}(G)$
to $[0,\infty)$) is Borel, since $H\mapsto m_{f}(H)$ is continuous
and, for a fixed compact subset $C$, the evaluation map $\mu\mapsto\mu(C)$
from $\mathcal{M}(G)$ to $[0,\infty)$ is Borel. Thus, the map 
\begin{align*}
\mathcal{C} & \longrightarrow(0,\infty)\\
H & \longmapsto {m_{f}(H)(H)}
\end{align*}
is Borel, since $m_{f}(H)(H)=\sup_{n}m_{f}(H)(K_{n}),$ and the supremum
of a countable collection of Borel maps is still Borel. Now, for $H\in\mathcal{C}$,
we have 
\[
m_{H}=\frac{1}{m_{f}(H)(H)}\,m_{f}(H),
\]
so $H\mapsto m_{H}$ is Borel. 
\end{proof}

Recall that an element $g\in G$ is called \emph{elliptic} if $\overline{\langle g\rangle}$ is compact. Given a compact IRS $\mu$, we define a conjugation invariant probability measure $\nu$ on $G$ that is supported on elliptic elements.
\begin{proposition}\label{AveragingMu}
Let $\mu$ be a compact IRS of an lcsc group $G$. Then the formula for $\nu$ on $G$ given by, for every Borel $A\subseteq G$,
\[
\nu(A)=\int_{\Sub(G)}m_{H}(A\cap H)\mathrm{d}\mu(H)
\]
defines a conjugation invariant probability measure on $G$.
Moreover, $\bigcup\supp\mu\subseteq\supp\nu$.
\end{proposition}
\begin{proof}
By Lemma \ref{Borel}, for any Borel $A\subseteq G$, the map $H\mapsto m_{H}(A)=m_{H}(A\cap H)$
is Borel. Therefore, we may define $\nu(A)=\int_{\Sub(G)}m_{H}(A\cap H)\mathrm{d}\mu.$
It is conjugation invariant since $\mu$ is conjugation invariant.
Now, suppose $g\notin\supp\nu$, and let us prove $g\notin\bigcup\supp\mu$.
The fact $g\notin\supp\nu$ means there is an open neighbourhood
$V$ of $g$ such that $\nu(V)=0$. This means that $\mu(\mathrm{Hit}(V))=0$: If $\mu(\mathrm{Hit}(V))$ were positive, then $\nu(V)=\int_{\mathrm{Sub}(G)}m_{H}(V\cap H)$
would have been positive, since $m_{H}(V\cap H)>0$ for every $H\in\mathrm{Hit}(V)$.
Thus, $\mu(\mathrm{Miss}(V))=1$. The latter is a closed subset of measure $1$, so it contains $\supp(\mu)$. Since $g\in V$, this means  $g\notin\bigcup\supp\mu$.
\end{proof}

\begin{remark}
    Conversely, given a conjugation invariant probability measure $\nu$ on $G$ that is supported on the set of elliptic elements, the map $G \to \Sub(G), g \mapsto \overline{\langle g \rangle}$ defines a compact IRS. It is not hard to check that this map is measurable, and Hofmann--Willis show that it is continuous if and only if $G$ is totally disconnected \cite{HofmannWillis2015}. 
\end{remark}

The following generalization of a lemma by U\v{s}akov \cite{usakov1970classes} and Wang \cite{wang1971compactness} can be found in \cite{Caprace_Goffer_Lederle_Tsankov_2025}.
\begin{lemma}\label{Usakov}
    Let $H\le G$ be a compact subgroup such that $\{gHg^{-1}\mid g\in G\}$ is a closed subset of $\Sub(G)$. Then $\{gHg^{-1}\mid g\in G\}$ is contained in a compact normal subgroup of $G$.
\end{lemma}

Whenever $N \triangleleft G$ is a normal subgroup, the quotient map $G \to G/N$ induces a map $\Sub(G) \to \Sub(G/N)$ sending $H$ to the closure of $HN/N$; it is in general not continuous (unless $N$ is compact), but measurable. Given an IRS $\mu$ of $G$, it pushes forward to an IRS $\mu/N$ of $G/N$ and if $\mu$ is compact, clearly also $\mu/N$ is.

\begin{lemma}\label{lem:compact normal quotient}
    Let $G$ be an lcsc group and $N \triangleleft G$ a compact normal subgroup. Let $\mu$ be a compact IRS of $G$. If $\mu/N$ is contained in a compact normal subgroup, then $\mu$ is contained in a compact normal subgroup.
\end{lemma}

\begin{proof}
    Let $KN/N \triangleleft G/N$ be a compact normal subgroup containing $\mu/N$. Then $KN$ is a compact normal subgroup of $G$ containing $\mu$.
\end{proof}

\section{Lie groups and almost connected groups}

Recall that an lcsc group $G$ is called \emph{almost connected} if $G/G^0$ is compact, where $G^0$ is the connected component of the identity.
In this case, we have that $\calW(G) = \calE(G)$ is a compact normal subgroup of $G$, see \cite[Example 2.3]{Cornulier_commability}.
Recall from the famous Gleason--Yamabe theorem that any
almost
connected
locally compact
group is pro-Lie; so most of this section deals with Lie groups.

\begin{theorem}\label{thm:Lie}
    In a Lie group, every ergodic compact IRS is contained in a compact normal subgroup.
\end{theorem}

The proof is a combination of three results from the literature: First, there are only countably many conjugacy classes of compact subgroups, second, every amenable IRS is contained in the amenable radical and third, every cofinite subgroup of an amenable Lie group is cocompact.

\begin{theorem}\label{thm:Lie gp has countably many subgroups up to conj}
    Lie groups admit only countably many conjugacy classes of compact subgroups.
\end{theorem}
\begin{proof}
    Let $G$ be an almost connected Lie group, and let $K$ be a maximal compact subgroup of $G$. Since $K$ is a compact Lie group, it admits only countably many closed subgroups (see \cite[Corollary 1.7.27]{Palais1960} or \cite[Proposition 2.4]{Gas}). Since every compact subgroup of $G$ is conjugate
    to a subgroup of $K$, this proves the claim for almost connected Lie groups. If $G$ is any Lie group, then every compact subgroup $K$ is contained in $KG^0$, which is an almost connected open subgroup of $G$. Since $G$ admits only countably many open subgroups, the claim follows.
\end{proof}

\begin{corollary}
\label{cor:sing-conj}Let $G$ be a Lie group and let $\mu$ be an ergodic compact IRS of $G$. Then 
there exists a compact subgroup $K\le G$ for which $\mu\left(\left\{ gKg^{-1}\middle|g\in G\right\} \right)=1.$
\end{corollary}
\begin{proof}
By Theorem \ref{thm:Lie gp has countably many subgroups up to conj}, there exists a countable collection $\left\{ K_{i}\right\} _{i\in\mathbb{N}}$ of compact subgroups, such that every compact subgroup of $G$ is conjugate
to $K_{i}$ for some $i\in\mathbb{N}$. Since $\mu$
is a compact IRS, we have:
\[
\mu\left(\bigcup_{i\in\mathbb{N}}\left\{ gK_{i}g{}^{-1}\middle|g\in G\right\} \right)=1.
\]
Since $\mu$ is ergodic, there exists $i_0\in\mathbb{N}$ such that
\(
\mu(\left\{ gK_{i_0}g{}^{-1}\middle|g\in G\right\}) =1.
\)
\end{proof}

\begin{lemma}\label{lem:AmenableLie}\label{lem:amenable-Lie}
Let $G$ be an amenable Lie group. Then every ergodic compact IRS is contained in a compact normal subgroup.
\end{lemma}

\begin{proof}
Let $\mu$ be an ergodic compact IRS of an amenable Lie group $G$.
By Corollary \ref{cor:sing-conj}, there is a compact subgroup
$K\le G$, such that $\mu(\{gKg^{-1} ~|~g\in G\})=1$. By Lemma \ref{lem:supp-confinite}, its normalizer $N_G(K)$ is cofinite in $G$.
By a result by Bader--Caprace--Gelander--Mozes, \cite[Corollary 5.25]{PropertyM},
every cofinite subgroup of an amenable Lie group is cocompact. The map
\begin{align*}
G/N_G(K) & \to\mathrm{Sub}(G)\\
gN_G(K) & \mapsto gKg{}^{-1}
\end{align*}
is continuous and $G/N_G(K)$ is compact, so $\left\{ gKg^{-1}\middle|g\in G\right\} $
is compact in $\Sub(G)$, hence closed. It follows from Lemma \ref{Usakov} that there is a compact normal subgroup $N$ containing $\{gKg^{-1}\mid g\in G\}=\bigcup\supp(\mu)$. Since $N$ is closed,  $\ncm \le N$, as needed. 
\end{proof}

\begin{proof}[Proof of Theorem \ref{thm:Lie}]
Let $\mu$ be an ergodic compact IRS of a Lie group $G$.
By Corollary \ref{cor:sing-conj} there exists a compact subgroup $K \le G$ such that $\mu$ is supported on the conjugacy class $[K]_G$ of $K$.
By \cite{bader2016amenable}, every amenable, and in particular every compact, IRS is contained in the amenable radical $R$. But also $R$ is a Lie group and has only countably many conjugacy classes of compact subgroups.
So there exist countably many $g_i \in G$ with
$[K]_G = \bigsqcup_i [g_i K g_i^{-1}]_R$,
and therefore $\mu$ decomposes as a direct sum $\mu = \sum_{i \in \bbN} \alpha_i \mu_i$, where each $\mu_i$ is a compact IRS of $R$ supported on the $R$-conjugacy class $[g_i K g_i^{-1}]_R$ and, in particular, ergodic.
But $R$ is normal and hence $[g_i K g_i^{-1}]_R = g_i [K]_R g_i^{-1}$, which by invariance of $\mu$ means that these $R$-conjugacy classes all have the same, positive, measure.
In particular, there can only be finitely many.
Now $\bigcup \supp(\mu) = \bigcup_i \supp(\mu_i)$, which is a finite union of sets with compact closure by Lemma \ref{lem:amenable-Lie}.
By Lemma \ref{Usakov}, we are done.
\end{proof}

\begin{corollary}\label{Cor:AlmostConnected}
    Let $G$ be an lcsc group such that $G/G^0$ is compact. Then
    every compact IRS of $G$ is contained in the polycompact radical, which is compact.
\end{corollary}

\begin{proof}
    By Yamabe's theorem, there exists a compact normal subgroup $N \triangleleft G$ such that $G/N$ is a Lie group. Now the result follows from Theorem \ref{thm:Lie} and Lemma \ref{lem:compact normal quotient}.
\end{proof}

\begin{corollary}
    Let $G$ be an lcsc group such that $G/G^0$ is discrete. Then every ergodic compact IRS of $G$ is contained in a compact normal subgroup.
\end{corollary}

\begin{proof}
    Let $\mu$ be an ergodic compact IRS of $G$.
    By Theorem \ref{thm:Lie}, the quotient IRS $\mu/G^0$ is contained in a finite, normal subgroup $F \triangleleft G/G^0$. Then $\mu$ is an IRS of the pre-image $N$ of $F$. Since $G^0 \le N$ is cocompact, the group $N$ is almost connected and $\mu$ is an IRS of $N$. By Corollary \ref{Cor:AlmostConnected}, the compact IRS $\mu$ is contained in the compact group $\calW(N)$. 
\end{proof}

\section{Totally disconnected groups} \label{sect:tdlc}

For
brevity, we call a topological group \emph{tdlc} if it is totally disconnected and locally compact, and \emph{tdlcsc} if it is moreover second-countable.

\begin{definition}[Willis]
    Let $G$ be a tdlc group and $g \in G$. 
    The \emph{parabolic subgroup} of $g$ is\footnote{Equality of these sets is proved in \cite[Prop.\ 3]{Willis1994}.}
    \begin{align*}
        \mathrm{par}(g) &:= \{h \in G \mid \{ g^n h g^{-n}  \mid n \geq 0\} \text{ has an accumulation point}\} \\
        &= \{h \in G \mid \overline{\{ g^n h g^{-n}  \mid n \geq 0\}} \text{ is compact}\}.
    \end{align*}
    The \emph{Levi subgroup} of $g$ is
    \[
    \mathrm{lev}(g) := \mathrm{par}(g) \cap \mathrm{par}(g^{-1})
       = \{h \in G \mid \overline{\{g^n h g^{-n} \mid n \in \bbZ\}} \text{ is compact}\}.
    \]
\end{definition}

\noindent
We define $\lev(G) := \bigcap_{g \in G} \lev(g)$; this subgroup was introduced in \cite{Willis2023} and called the \emph{approximate center}, denoted $AZ(G)$.
Both the parabolic and the Levi subgroup of an element are closed (see \cite[Prop. 3]{Willis1994}). It follows that $\lev(G)$ is a closed, normal subgroup.
We always have
\begin{equation}
\calW(G)\le \lev(G).
\end{equation}
To see this, write $I_g(x):=gxg^{-1}$
for $g,x\in G$.
For each compact normal subgroup~$K$ of~$G$,
given $x\in K$ and $g\in G$
we have $\{I_g^n(x)\colon n\in \bbZ\} \subseteq K$, whence $x\in \lev(g)$
and thus $x\in \lev(G)$. Hence $K\subseteq \lev(G)$,
from which the assertion follows.

\begin{theorem}\label{thm:levi}
    Let $\mu$ be a compact IRS of a tdlcsc group $G$. Then $\ncm \le \lev(G)$.
\end{theorem}

\begin{proof}
    We consider the probability measure $\nu$ defined in Proposition \ref{AveragingMu}. It is a well-defined, conjugation invariant probability measure on $G$, and $\bigcup\supp(\mu)\subseteq\supp(\nu)$. Since $\nu$ is conjugation invariant, for every $g \in G$ we have by Poincar\'e recurrence
    \[
    \nu( \{ h\in G \mid g^n h g^{-n}\text{ accumulates at }h\}) = 1.
    \]
    In particular $\nu(\mathrm{par}(g))=1$. Since $\mathrm{par}(g)$ is closed, this means that $\bigcup\supp(\mu)\subseteq\supp(\nu)\subseteq\mathrm{par}(g)$. 
    Since $g$ was arbitrary, we have $\bigcup\supp(\mu)\subseteq\lev(G)$. Since $\lev(G)$ is a closed subgroup, we get that the normal closure of $\mu$ is contained in $\lev(G)$.
\end{proof}
\subsection{\texorpdfstring{$p$}{p}-adic Lie Groups}
We now turn to $p$-adic Lie groups and observe:\footnote{We mention that second countability is inessential for Remark~\ref{closedpadic} through Lemma~\ref{glo2}.}

\begin{remark}\label{closedpadic}
$\calW(G)$ is closed for each $p$-adic Lie group~$G$.
\end{remark}

In fact, $\calW(G)=\bigcup\calK$
for the set $\calK$ of compact normal subgroups
of~$G$ and the following lemma applies,
which generalizes the fact
that an ascending union $\bigcup_{n\in \bbN}H_n$
of closed subgroups $H_1\subseteq H_2\subseteq\cdots$
of a $p$-adic Lie group is always closed
(see Lemma~6.5 in \cite{MAT}; cf.\ step~1 of the proof of Theorem 3.5 in \cite{Wan}).

\begin{lemma}
Let $G$ be a $p$-adic Lie group
and $\calH$ be a set of closed subgroups of~$G$
which is directed under $\subseteq$.
Then $H:=\bigcup\calH$ is closed in~$G$.
\end{lemma}

\begin{proof}
If $x$ is in the closure of~$H$,
there exists a sequence $(x_n)_{n\in \bbN}$
in $H$ which converges to~$x$. Then $x_n\in H_n$ for some $H_n\in \calH$; as $\calH$ is directed, we may assume that $H_1\subseteq H_2\subseteq\cdots$.
Since $C:=\bigcup_{n\in\bbN}H_n$
is closed, $x\in C\subseteq H$ follows.
\end{proof}

A locally compact group~$G$
is called \emph{pro-discrete} if each
identity neighbourhood of~$G$ contains
an open, normal subgroup of~$G$.

\begin{theorem}\label{thm:Gloeckner W(G) open in lev(G)}
If $G$ is a compactly generated $p$-adic
Lie group, then each identity neighbourhood
of $\lev(G)$
contains a compact open subgroup of~$\lev(G)$
which is normal in~$G$.
In particular, $\lev(G)$ is pro-discrete
and $\calW(G)$
is an open subgroup of~$\lev(G)$.
\end{theorem}

The following two lemmas
will help us to prove the theorem.

\begin{lemma}\label{glo1}
$\alpha\in \GL_n({\mathbb Q}_p)$
is an elliptic element if and only if
$|\lambda|=1$ for each
eigenvalue~$\lambda$
of~$\alpha$ in an algebraic closure
$\overline{{\mathbb Q}}_p$
of ${\mathbb Q}_p$,
where $|\cdot|\colon \overline{{\mathbb Q}}_p\to\bbR$
is the unique extension of the $p$-adic
absolute value on ${\mathbb Q}_p$
to an absolute value on $\overline{{\mathbb Q}}_p$.
\end{lemma}

\begin{proof}
If $|\lambda|=1$ for all $\lambda$,
then $\alpha$ is elliptic by Lemma~3.9 in \cite{MAT}.
Conversely, let
$\alpha$ be elliptic and $\bbK\subseteq \overline{{\mathbb Q}}_p$
be a splitting field for the characteristic polynomial
of~$\alpha$. If $|\lambda|>1$ for some
eigenvalue of~$\alpha$ in~$\bbK$ and $v\in \bbK^n$ a corresponding eigenvector, then the orbit
\[
\alpha^{\bbN_0}(v):=\{\alpha^m(v)\colon m\in \bbN_0\}=\{\lambda^mv\colon m\in \bbN_0\}
\]
is unbounded. But $\alpha^{\bbN_0}(v)$
is relatively compact and hence bounded
as $\alpha^{\bbN_0}$ is contained in
a compact subset of $\GL_n({\mathbb Q}_p)$
and the action $\GL_n({\mathbb Q}_p)\times \bbK^n\to\bbK^n$,
$(\beta,x)\mapsto \beta(x)$ is continuous; contradiction.
Using negative powers of $\alpha$ instead,
we see that also
$|\lambda|<1$
cannot occur.
\end{proof}

\noindent
Given a $p$-adic Lie group~$G$,
we let $\Lie(G)$ be its Lie algebra.
If $\psi\colon G\to H$
is an analytic group homomorphism
between $p$-adic Lie groups, then $\Lie(\psi)\colon \Lie(G)\to \Lie(H)$
denotes the associated Lie algebra homomorphism.
If $M$ is an analytic $p$-adic manifold,
we let $TM$ be its tangent bundle.
If $f\colon M\to N$
is an analytic map between $p$-adic manifolds,
we write $Tf\colon TM\to TN$ for its tangent map.
Moreover,
$\Aut(G)$ denotes the group of
automorphisms of a $p$-adic Lie group~$G$.

\begin{lemma}\label{glo2}
Let $H$ and $N$ be $p$-adic Lie groups, $\frakn:=\Lie(N)$ and
\[
\alpha\colon H\to\Aut(N),\quad h\mapsto\alpha_h
\]
be a homomorphism of groups such that the associated action
$\widehat{\alpha}\colon H\times N\to N$, $(h,x)\mapsto\alpha_h(x)$
is analytic. Assume that $H$ is compactly generated
and that $\Lie(\alpha_h)\colon \frakn\to \frakn$
is an elliptic element of $\GL(\frakn)$
for each $h\in H$. Then each identity neighbourhood
$U\subseteq N$ contains a compact open subgroup
$V$ of~$N$ such that
\[
\alpha_h(V)=V\quad\mbox{for all $\, h\in H$.}
\]
If the inner automorphism
$I^N_x\colon N\to N$, $y\mapsto xyx^{-1}$
is in $\alpha(H)$ for all $x\in N$,
then each $V$ as before is a normal subgroup of~$N$
and $N$ is pro-discrete.
\end{lemma}
\begin{proof}
Let $K\subseteq H$ be a compact generating set.
Since $\widehat{\alpha}$ is analytic, so
is $T\widehat{\alpha}\colon TH\times TN\cong T(H\times N)\to TN$
and hence also its restriction $H\times \frakn\to\frakn$, $(h,x)\mapsto \Lie(\alpha_h)(x)$.
Notably, $h\mapsto \Lie(\alpha_h)(x)$ is analytic for $x$ in a basis of $\frakn$,
entailing that the group homomorphism
\[
\beta\colon H\to\GL(\frakn),\quad  h\mapsto \beta_h:=\Lie(\alpha_h)
\]
is analytic and hence continuous.
As a consequence, $S:=\beta(H)$
is a compactly generated subgroup of~$\GL(\frakn)$,
generated by $\beta(K)$.
By hypothesis, $\Lie(\alpha_h)$ is an elliptic element
for each $h\in H$.
Thus, by Parreau's Theorem
(see Th\'{e}or\`{e}me~1 in \cite{Par}; cf.\ Proposition~5.1 in \cite{UNI}),
the compactly generated subgroup
$S$ of $\GL(\frakn)$ is relatively compact.
There is a compact open $\bbZ_p$-submodule
$W\subseteq \frakn$ such that
\[
g(W)=W\quad\mbox{for all $\,g\in S$}
\]
(cf.\ Theorem 1 in
Part II, Chapter~IV, Appendix 1
of \cite{Ser}).
Then $\|0\|:=0$,
\[
\|x\|:=p^{-\max\{n\in \bbZ \colon x\in p^nW\}}
\]
for $0\not=x\in \frakn$
defines an ultrametric norm $\|\cdot\|\colon \frakn\to[0,\infty[$
such that $p^nW=\{x\in \frakn \colon \|x\|\le p^{-n}\}$
for each $n\in \bbZ$. By Proposition~3 in
Chapter~II, \S8, no.\,3 of \cite{Bou},
there is $n_0\in \bbZ$
such that the Baker-Campbell-Hausdorff series
converges to a group multiplication
$*\colon
p^nW\times p^nW\to p^nW$ for all integers $n\geq n_0$.
After increasing~$n_0$,
we may assume that there exists an
exponential function $\exp_N\colon p^{n_0}W\to N$
which is an analytic diffeomorphism
onto an open identity neighbourhood~$O\subseteq N$
(see Definition~1 in
Chapter~III, \S4, no.\,3 in \cite{Bou}).
After increasing~$n_0$,
we may assume that $O$
is a compact open subgroup of~$N$
and $\exp_N\colon (p^{n_0}W,*)\to O$
an isomorphism of $p$-adic Lie groups
(see Proposition~3 in
Chapter~III, \S7, no.\,2 in \cite{Bou}).
Then also
\[
U_n := \exp_N(p^nW)
\]
is a compact open subgroup
of~$N$ for each $n\geq n_0$ and $\exp_N$
restricts to an isomorphism of $p^nW$ onto~$U_n$.
Note that each continuous homomorphism
\[
\gamma\colon \bbZ_p\to O
\]
is uniquely determined by $\dot{\gamma}(0):=\frac{d}{dz}\Big|_{z=0}\gamma(z)$.
In fact, $\gamma(1)=\exp_N(x)$ for some $x\in p^{n_0}W$.
Then $\gamma(n)=\gamma(1)^n=\exp_N(x)^n=\exp_N(nx)$
for all $n\in \bbZ$, whence
\begin{equation}\label{gammaexp}
\gamma(z)=\exp_N(zx)\quad\mbox{for all $z\in\bbZ_p$}
\end{equation}
by continuity. Thus $\dot{\gamma}(0)=\frac{d}{dz}
\Big|_{z=0}
\exp_N(zx)=x$, which determines $\gamma$
via~(\ref{gammaexp}).

The map
$\widehat{\beta}\colon H\times\frakn\to\frakn$, $(h,x)\mapsto \beta_h(x)$
is continuous and $\widehat{\beta}(h,x)=0$
holds for all $(h,x)$ in the compact set $K\times\{0\}$.
Hence,
there is $m_0\geq n_0$ such that
\[
\widehat{\beta}(K\times p^{m_0}W)\subseteq p^{n_0}W.
\]
Likewise, we get
$\widehat{\alpha}(K\times \exp_N(p^{m_0}W))\subseteq O$,
after increasing~$m_0$.

Then
\begin{equation}\label{compactnaturality}
\exp_N(\beta_h(x))=\alpha_h(\exp_N(x))
\quad\mbox{for all $h\in K$ and $x\in p^{m_0}W$.}
\end{equation}
In fact,
both of the maps
$\bbZ_p\to O$ taking $z\in \bbZ_p$ to
$\exp_N(\beta_h(zx))$ and $\alpha_h(\exp_N(zx))$,
respectively, are continuous homomorphisms
with derivative $\beta_h(x)$ at $z=0$.
They therefore coincide and $z:=1$ yields (\ref{compactnaturality}).
For all $n\geq m_0$,
we deduce for each $h\in K$ that
$\alpha_h(\exp_N(p^nW))=\exp_N(\beta_h(p^nW))=\exp_N(p^nW)$,
using that $\beta_h\in S$. Thus
$\alpha_h(U_n)=U_n$
for all $h\in K$, whence $\alpha_h(U_n)=U_n$
for all $h\in H$ as $K$ generates~$H$.
Choosing $n\geq m_0$ large enough, we have
$U_n\subseteq U$. Then $V:=U_n$
is as desired.
\end{proof}

\noindent
Generalizing the case of inner automorphisms,
if $\alpha\colon G\to G$
is an automorphism
of a totally disconnected,
locally compact group~$G$,
we write $\lev(\alpha)$
for its Levi factor, viz.\
the set of all elements $x\in G$ such that
$\alpha^\bbZ(x)$ is relatively compact in~$G$.
Then $\lev(\alpha)$ is a closed
subgroup of~$G$ (cf.\ \cite{Willis1994}).
We are now ready to prove
Theorem~\ref{thm:Gloeckner W(G) open in lev(G)}.

\begin{proof}[Proof of Theorem~\ref{thm:Gloeckner W(G) open in lev(G)}]
Consider the group homomorphism
$I\colon G\to\Aut(G)$,
$g\mapsto I_g$.
The associated action $\widehat{I}\colon G\times G\to G$, $(g,x)\mapsto I_g(x)=gxg^{-1}$
is analytic.
Since $A:=\lev(G)$ is a normal subgroup of~$G$,
we obtain a group homomorphism
\[
\alpha\colon G\to\Aut(A),\quad g\mapsto \alpha_g:=I_g|_A^A.
\]
Being a closed normal subgroup,
$A$ is a Lie subgroup of~$G$,
entailing that the restriction
\[
\widehat{\alpha}\colon G\times A\to A, \quad (g,x)\mapsto \alpha_g(x)
\]
of $\widehat{I}$ to $G\times A$ is analytic as well.
Let $\frakg:=\Lie(G)$ be the Lie algebra of~$G$.
We consider $\fraka:=\Lie(A)$ as a Lie subalgebra
of~$\frakg$ and likewise $\frakl_g:=\Lie(\lev(g))$
for $g\in G$.
Consider the adjoint representation
$\Ad\colon G\to\Aut(\frakg)$,
$g\mapsto \Ad_g:=\Lie(I_g)$
of~$G$.
Since $A$ is a closed normal subgroup of~$G$,
we have
$\Ad_g(\fraka)=\fraka$ for all $g\in G$.
We therefore obtain a group homomorphism
\[
\beta\colon G\to \Aut(\fraka),\quad g\mapsto \beta_g:=\Ad_g|_{\fraka}^{\fraka}.
\]
The identification of $\Lie(A)$ with a vector subspace of~$\Lie(G)$
implies that $\beta_g=\Lie(\alpha_g)$.\\[2.3mm]
It is known that
$\Lie(\lev(I_g))=\lev(\Ad_g)$
holds for $\Ad_g$ as an automorphism
of the tdlcsc group
$(\frakg,+)$;
see Theorem 3.5\,(iv) in \cite{Wan}.
Thus
\[
\beta_g^\bbZ(x)=\Ad_g^\bbZ(x)
\]
is relatively compact
in~$\frakg$ (and hence in~$\fraka$)
for each $x\in \fraka\subseteq \lev(\Ad_g)$.
Let $\bbK\subseteq\overline{{\mathbb Q}}_p$ be a splitting field
for the characteristic polynomial of~$\Ad_g$.
By Theorem~3.6\,(b) in~\cite{MAT},
all eigenvalues of the automorphism
$(\beta_g)_\bbK:=\beta_g\otimes_{{\mathbb Q}_p}\text{id}_\bbK$
of $\fraka_\bbK:=\fraka\otimes_{{\mathbb Q}_p}\bbK$
must have absolute value~$1$ in $\overline{{\mathbb Q}}_p$.
Thus $\beta_g$ is an elliptic element of $\GL(\fraka)$,
by Lemma~\ref{glo1}.
Applying Lemma~\ref{glo2} to $H:=G$
and $N:=A$,
we conclude that each identity neighbourhood
$U$ of~$A$ contains a compact open subgroup~$V$
of~$A$ which is invariant under
$\alpha_g=I^G_g|_A^A$ for all $g\in G$.
Since $I^G_g|_A^A=I^A_g$ for all $g\in A$,
we see that~$V$ is, in particular,
a normal subgroup of~$A$.
\end{proof}

\begin{proof}[Proof of Theorem \ref{thm:p-adic}\ref{cg-or-alg}, compactly generated case]
    By Theorem \ref{thm:Gloeckner W(G) open in lev(G)}, $\lev(G)$
    has a compact,
    open subgroup $N$ which is normal
    in~$G$.
    Let $\mu$ be an ergodic compact IRS of $G$. 
    By Theorem \ref{thm:levi}, $\ncm$ is contained in $\lev(G)$. Therefore, we can push $\mu$ forward to an IRS $\mu/N$ on the countable group $\lev(G)/N$ that is ergodic under the $G$-action; since $
    \mu$ is almost surely compact and $
    \mathrm{lev}(G)/N$ is discrete, $\mu/N$ is almost surely finite. Since $\lev(G)/N$ is countable, it admits only countably many finite subgroups; since $\mu/N$ is ergodic under the $G$-action on $\mathrm{Sub}(\lev(G)/N)$, it must be supported on a finite $G$-conjugacy class. It follows from  U\v{s}akov's Lemma (Lemma \ref{Usakov}) that there is a finite, normal subgroup $F$ of $\lev(G)/N$ such that $\mu/N$ is contained in $F$. 
    We conclude with Lemma \ref{lem:compact normal quotient}.
\end{proof}

In Section \ref{reduction}, we will prove that a tdlcsc group satisfies Conclusion \ref{weak} if and only if every compactly generated subgroup of it does (Corollary \ref{cor:reduction of weak conc to cg}), obtaining Theorem \ref{thm:p-adic}\ref{general-p-adic}.

\subsection{Algebraic groups}

\begin{theorem}[Theorem \ref{thm:p-adic}\ref{cg-or-alg}, algebraic case]\label{thm:algebraic groups}
    Let $\mathbf{G}$ be a connected algebraic group defined over $k$, a local field of characteristic zero, and let $G=\mathbf{G}(k)$. Then every ergodic compact IRS of $G$ is contained in a compact normal subgroup.
\end{theorem}

\begin{proof}
    We consider the probability measure $\nu$ defined in Proposition \ref{AveragingMu}. By \cite[Theorem 3.1]{FinitCentralMeasures}, every conjugation-invariant probability measure on $G$ is supported on $B(G)$, the subgroup of elements whose conjugacy class has compact closure.\footnote{A note on terminology: Sit uses the term `central' rather than `conjugation-invariant', and opts to talk about finite measures rather than probability measures.}
Since $\bigcup\supp(\mu)\subseteq\supp(\nu) $, it follows that $\bigcup\supp(\mu)$ is contained in $B(G)$. We already know $\bigcup\supp(\mu)$ is contained in $P(G)$, the subset 
of elliptic elements.
Observe that $B(G)\cap P(G)=\calW(G)$ (\cite[Proposition 2.4(2)]{Cornulier_commability}). Recalling Remark~\ref{closedpadic}, it follows that $\ncm\le \overline{\calW(G)}=\calW(G)$, as needed. 
\end{proof}

\section{Groups acting on metric spaces}

Throughout this
section, $(X,d)$ denotes a metric space on which an lcsc group $G$ acts continuously by isometries.
For a subset $A \subset X$, its \emph{diameter} is $\diam(A) := \sup_{x,y \in A} d(x,y) \in [0,\infty]$; the set $A$ is \emph{bounded} if its diameter is finite.
Recall that the action of $G$ is called \emph{cobounded} if there exists a bounded set $B \subset X$ such that $X = \bigcup_{g \in G} gB$.
The action is called \emph{proper} if for every bounded $B \subset X$ the set $\{g \in G \mid gB \cap B \neq \emptyset\}$ has compact closure.

\begin{lemma}\label{lem:r-bounds some point}
    Let $\mu$ be an ergodic compact IRS of $G$ and let $x \in X$.
    Then there exists $R > 0$ such that for $\mu$-almost every $H \le G$ there exists $g \in G$ with $\diam(H(gx)) \le R$.
\end{lemma}

\begin{proof}
    Since the action of $G$ is continuous, for every compact $H \le G$ and every $x \in X$ we have $\diam(Hx) < \infty$.
 Note that the map $\Sub(G)\to [0,\infty]$, $H\mapsto \diam(Hx)$ is measurable. Hence, for every $x \in X$,
    $$1=\mu\Big(\bigcup _{n\in \mathbb{N}}\{H ~|~\diam(Hx)<n\}\Big).$$ It follows that there exists $n>0$ for which $\mu(\{H~|~\diam(Hx)<n\})>0$. Since $\mu$ is ergodic, $\mu(\{H~|~\diam(Hgx)<n\text{, for some $g\in G$}\})=1$, as claimed.  
\end{proof}

\begin{corollary}\label{cor:big balls}
    Suppose the action of $G$ on $X$ is cobounded, and 
    let $\mu$ be an ergodic compact IRS of $G$. Then there exists $R > 0$ such that for any subset $A \subset X$ containing arbitrarily large balls,
    for $\mu$-almost every $H \le G$, the set $\{x \in A ~|~ \diam(Hx)\le R\}$ is unbounded.
\end{corollary}

\begin{proof}
    Fix $x\in X$ and let $R>0$ be as given by Lemma \ref{lem:r-bounds some point}.
    We first show that, $\mu$-almost surely, the set $\{x\in A~|~\diam(Hx)\le R\}$ is non-empty. 
    For a subgroup $H$, denote $H_R=\{y\in X ~|~\diam(Hy)\le R\}$. By the choice of $R$, we have $H_R\neq \emptyset$ for $\mu$-almost every $H$. It follows that $$\mu\Big( \bigcup_{n>0}\{ H ~|~H_R\cap B(x,n)\neq \emptyset \} \Big)=1,$$
    and the union is increasing. Let $\epsilon>0$, so there exists $n>0$ for which $\mu\big(\{ H ~|~H_R\cap B(x,n)\neq \emptyset \} \big)>1-\epsilon$. By assuming $n$ is large enough, and since the action is cobounded, there exists $g_n\in G$ such that $B(g_nx,n)\subset A$. Since $\mu$ is invariant, we get $\mu\big(\{ H ~|~H_R\cap B(g_nx,n)\neq \emptyset \} \big)>1-\epsilon$, and in particular, $$\mu\big(\{ H ~|~H_R\cap A\neq \emptyset \} \big)>1-\epsilon.$$ As $\epsilon$ was arbitrary, $\mu\big(\{ H ~|~H_R\cap A\neq \emptyset \} \big)=1$.

    Now, note that for every $N \in \bbN$ the set $A_N := A \setminus B(x,N)$ still contains arbitrarily large balls, and since a countable intersection of sets of full measure
    has full measure, $\mu$-almost every $H \le G$ satisfies $H_R \cap A_N \neq \emptyset$ for all $N$. This finishes the proof. 
\end{proof}

\subsection{Hyperbolic spaces}

The group $G$ is called \emph{hyperbolic} if it acts properly and coboundedly on a proper geodesic Gromov hyperbolic space $(X,d)$. However, in order to also encompass the case of infinitely ended groups, we do not assume properness of the space $X$ in this subsection.
We refer to \cite[Section 3]{caprace2015amenable} for basics.
In particular, the Gromov boundary $\partial X$ is not defined via equivalence classes of geodesic rays, but equivalence classes of Cauchy--Gromov sequences (which we do not define here).
The reader only interested in proper spaces can think of a Cauchy--Gromov sequence $(x_n)$ as points $c(n)$, where $c \colon [0,\infty) \to X$ is a geodesic ray.
Since the action of $G$ on $X$ is by isometries, it induces a continuous action on $\partial X$.
The action of $G$ on $X$ is called \emph{non-elementary} if there exist two hyperbolic elements in $G$ whose pairs of fixed points on the boundary are not the same (it is still possible that all hyperbolic elements have one common fixed point; in that case the hyperbolic group is called \emph{focal}).

The goal of this subsection is to prove the following theorem, of which Theorem \ref{thm:hyp or ends} is a special case: \begin{theorem}\label{thm:hyperbolic}
    Let $G$ be an lcsc group acting non-elementarily, properly and coboundedly on a (not necessarily proper) geodesic hyperbolic space $(X,d)$.
    Then the polycompact radical $\calW(G)$ is compact and contains every compact IRS.
\end{theorem}

The proof of Theorem \ref{thm:hyperbolic} is given
after the following lemma.

\begin{lemma}\label{lem:fixes boundary point}
    Let $G$ be an lcsc group acting coboundedly on a hyperbolic space $(X,d)$.
    Let $\xi\in \partial X$ and let $\mu$ be an ergodic compact IRS of $G$. 
    Then $\mu$-almost every subgroup fixes $\xi$.
\end{lemma}

\begin{proof}
Choose a basepoint $x \in X$.
It follows from the definitions that, given Cauchy--Gromov sequences $(x_n),(y_n)$ satisfying $d(x_n,x), d(y_n,x) \geq n$,
if they do not represent the same boundary point then  $d(x_n,y_n) \geq n$ for all large enough $n\in \mathbb{N}$.

Let $(x_n)$ be a Cauchy--Gromov sequence representing $\xi$.
Consider the set $A=\bigcup_{n\ge 0}B(x_n,\frac{1}{4}n)$. Note that $A$ contains arbitrarily large balls, and therefore, by Corollary \ref{cor:big balls} there exists $R > 0$ such that for $\mu$-almost every subgroup $H$ the set $H_R := \{x\in A ~|~ \diam(Hx)\le R \}$ is unbounded.
In particular, for $\mu$-almost all $H$ there exist infinitely many $k$'s such that, for every $h \in H$ we have 
$d(x_k,hx_k) \le d(x_k,x) + d(x,hx) + d(hx,hx_k) = 2 d(x_k,x) + d(x,hx) \le \frac{1}{2} k+R$, where $x \in H_R \cap B(x_k,\frac{1}{4}k)$.
In view of the first paragraph, this implies $h\xi=\xi$ for every $h\in H$.
\end{proof}

Note that hyperbolicity of $X$ is used only to guarantee that the distance between two non-equivalent sequences defining a boundary point grows faster than some function. Consequently, Lemma \ref{lem:fixes boundary point} holds for any metric space $X$, not necessarily hyperbolic, that satisfies this condition. 
We are now ready to prove Theorem \ref{thm:hyperbolic}.
\begin{proof}[Proof of Theorem \ref{thm:hyperbolic}.]
    Since the boundary of $X$ is Hausdorff and the $G$-action on it is continuous, stabilizers of points in $\partial X$ are closed. 
    It then follows from Lemma \ref{lem:fixes boundary point} that $\mu$ is contained in the kernel of the action on the boundary, which we denote by $K$.

    Let $B(G)$
    be the set of elements $b\in G$ whose conjugacy class $b^G=\{g^{-1}bg~|~g\in G\}$ has compact closure; it is a closed, normal subgroup, see \cite[Theorem 2.8]{Cornulier_commability}.
    We claim that
    \[
    B(G) = \{b \in G \mid \exists N > 0 \, \forall x \in X\colon d(bx,x) < N \}.
    \]
    Indeed, $\subset$ only uses coboundedness: Let $b \in B(G)$ and $y \in X$. There exists $R > 0$ such that for all $x \in X$ there is $g \in G$ with $d(gx,y) < R$.
    Then $d(bx,x) \le d(bx,bg^{-1}y) + d(bg^{-1}y,g^{-1}y) + d(g^{-1}y,x) \le 2 d(x,g^{-1}y) + d(gbg^{-1}y,y) \le 2R + \diam(b^G y)$,
    which is independent of $x$.
    Then $\supset$ only uses properness of the action: If $d(bx,x) < N$ for all $x \in X$ then,
    given $x \in X$ and $g \in G$ also $d(g b g^{-1}x,x) = d(bg^{-1}x,g^{-1}x) < N$.
    Hence $b^G \subset \{g \in G \mid B(x,N) \cap g(B(x,N)) \neq \emptyset \}$, which has compact closure.
    
    Denote by $P(G)$ the set of elliptic elements in $G$, and recall that $\calW(G) = B(G) \cap P(G)$. Since every $g \notin P(G)$ fixes exactly two boundary points, and in particular not the entire boundary, it is easy to see that $g \notin B(G)$ and so $B(G)\subset P(G)$. Therefore, we have $\calW(G) = B(G)$.
    Moreover, with similar arguments we get that $B(G)$ is the union of all normal subgroups that act on $X$ with bounded orbits.
    
    Since $B(G)$ is normal and the $G$-action on $X$ is non-elementary, \cite[Lemma 3.6]{caprace2015amenable} implies that: (i) the action of $B(G)$ on $X$ has bounded orbits -- in particular, $B(G)$ is compact; and (ii) $B(G)= K$ (indeed, since the action is cobounded $\partial_X(G) = \partial X$, using the notations in the reference).
    It follows that $\ncm \le K=B(G)=\calW(G)$, as claimed.
\end{proof}

A similar argument should work for most CAT(0) groups.

\subsection{Cayley--Abels graphs}

If $G$ is a compactly generated totally disconnected lcsc
group, and $U \le G$ a compact open subgroup, then $G/U$ is the vertex set of a locally finite, connected graph, on which $G$ acts by left multiplication. Such a graph is called a \emph{Cayley--Abels graph} for $G$.
Using that $G$ has only countably many compact open subgroups, and every compact subgroup is contained in a compact open subgroup, we see that for any ergodic compact IRS $\mu$ of $G$ there exists a Cayley--Abels graph in which $\mu$-almost every subgroup fixes a vertex.
One can push such an IRS to a shift-invariant probability measure on $\{0,1\}^{G/U}$ by mapping a subgroup to its (almost surely non-empty) set of fixed points.
Conversely, given any shift-invariant, ergodic probability measure on $\{0,1\}^{G/U}$ without an atom at $\{\emptyset\}$, the pointwise stabilizer of a random subset is an ergodic compact IRS of $G$.

Below is an example of this construction, for which we are able to show that the IRS is contained in the locally elliptic radical.

\begin{example}
    Let $G$ be a compactly generated totally disconnected lcsc group and $\Gamma$ a Cayley--Abels graph for $G$ with vertex set $V$. Fix $0 < p < 1$.
    We obtain an ergodic compact IRS $\mu$ of $G$ as the pointwise stabilizer of a random subset of $V$, where each vertex is chosen uniformly at random with probability $p$.
    We can show that $\ncm \le \calE(G)$:
    For $x,y \in V$ say $x \sim_n y$ if for every $m \geq n$ the $m$-balls around $x$ and $y$ coincide, and let $\sim = \bigcup \sim_n$ be the equivalence relation that is the union of the relations $\sim_n$. Note that for every $x \in V$ the $\sim_n$-equivalence class, denoted $[x]_n$, is finite, and $\Stab([x]) = \bigcup_{n} \Stab([x]_n)$ is thus an increasing union of compact groups and hence locally elliptic.
    Now if $x \nsim y$, there are infinitely many $(z_i)$ with $d(z_i,x) \neq d(z_i,y)$. Almost surely a $\mu$-random subgroup fixes at least one of the $z_i$'s and therefore $Hx \neq Hy$. Hence, for all $x \in V$ we get $\mu(\{\bigcup_n \Stab([x]_n)\}) = 1$ and in particular $\ncm \le \bigcap_{x \in V} \bigcup_n \Stab([x]_n)$. This (clearly normal) subgroup is an intersection of locally elliptic subgroups, so $\ncm \le \calE(G)$.
    
    Now let $F := \overline{\{g \in G \mid gx = x \text{ for all but finitely many } x \in V\}} \le \calW(G)$. It is not hard to see that $F \le \ncm$.
    Moreover, if $p > 1-p_c$, where $p_c$ is the critical vertex percolation probability of $\Gamma$, then $\ncm \le F$. We show in the subsequent proposition that $F$ is compact.
\end{example}

The fact that $F$ is compact is of independent interest and follows from a difficult theorem by Trofimov.
We mention
that the same theorem is the crucial ingredient in the proof that $\calW(G)$ is closed.

\begin{proposition}\label{prop:finitely supported compact}
    Let $\Gamma$ be a connected, locally finite, vertex transitive graph. Then the group of finitely supported automorphisms of $\Gamma$ has compact closure.
\end{proposition}

\begin{proof}
    Trofimov \cite[Corollary 4]{trofimov1992action} shows that there exists an equivalence relation $\sim$ on the set of vertices of $\Gamma$ that is compatible with the action of the automorphism group, i.e. $x \sim y$ if and only if $gx \sim gy$, has finite equivalence classes, and such that no bounded automorphism of $\Aut(\Gamma/{\sim})$ fixes a vertex of $\Gamma/{\sim}$.
    These properties ensure that $\Gamma/{\sim}$ is a Cayley--Abels graph of $\Aut(\Gamma)$. Any automorphism of $\Gamma$ fixing a vertex in $\Gamma$ also fixes its image in $\Gamma/{\sim}$. So every finitely supported automorphism of $\Gamma$ acts as a finitely supported and hence bounded automorphism of $\Gamma/{\sim}$ and hence lies in the kernel of the action of $\Aut(\Gamma)$ on $\Gamma/{\sim}$. This kernel is compact, which finishes the proof.
\end{proof}

\section{Reduction to compactly generated subgroups}\label{reduction}
The aim of this last section is to prove Corollary \ref{cor:reduction of weak conc to cg}, which gives a reduction of Conclusion~\ref{weak} to compactly generated subgroups. 

\begin{lemma}\label{lem:radicals and unions}
    Let $G$ be an lcsc group which is an ascending union
    $G = \bigcup_{n\in\bbN} O_n$
    of open subgroups.
    Then
        $\mathcal{E}(G) = \bigcup_{k \in \bbN} \bigcap_{n \geq k} \calE(O_n)$.
\end{lemma}

\begin{proof}
To show $\mathcal{E}(G) \le \bigcup_{k \in \bbN} \bigcap_{n \geq k} \calE(O_n)$, let $H \triangleleft G$ be locally elliptic. We have $H = \bigcup_{k \in \bbN} H \cap O_k$. Therefore, it is enough to show $H \cap O_k \le \bigcap_{n \geq k} \calE(O_n)$ for every $k\in\mathbb{N}$.
This is true because, for every $k$ we have $H \cap O_k \in \calE(O_k)$.
Therefore, if $n \geq k$, then $H \cap O_k \le H \cap O_n \le \calE(O_n)$ and we get what we need.

To see that $S:=\bigcup_{k \in \bbN} \bigcap_{n \geq k} \calE(O_n)\le \mathcal{E}(G)$,
we show that $S$ is normal and locally elliptic. As
the sequence $A_k:=\bigcap_{n \geq k} \calE(O_n)$ is ascending,
normality follows if each $g\in G$ normalizes $A_k$ for
large $k$. This holds for all $k\geq m$ with $m\in \bbN$ such that $g\in O_m$,
as $O_m\leq O_n$ and $\calE(O_n)$ is normal in $O_n$ for all $n\geq k$.

To prove local ellipticity, let $K\subseteq S$ be compact. Then $K\subseteq O_m$
for some $m$. We claim that $O_m\cap S =O_m\cap A_m$. If this is true,
then the subgroup $S$ is locally closed and hence closed in~$G$.
Moreover, $K\subseteq A_m\subseteq \calE(O_m)$. As $S$ is closed in~$G$, for the compact
closure in $\calE(O_m)$ we have $\overline{\langle K \rangle}\leq S$.

To prove the claim, let $k>m$.
Now $O_n\cap \calE(O_k)$
is a locally elliptic, normal subgroup of $O_n$
for all $n\in \{m,\ldots,k-1\}$, whence $\calE(O_n)\geq O_n\cap \calE(O_k)\geq O_m\cap \calE(O_k)\geq O_m\cap A_k$
and thus $O_m\cap A_m=O_m\cap A_k\cap \bigcap_{n=m}^{k-1}\calE(O_n)
= O_m\cap A_k$.
Hence $O_m\cap S=\bigcup_{k\geq m}O_m\cap A_k=O_m\cap A_m$.
\end{proof}

\begin{proposition}\label{prop:direct union}
    Let $G = \bigcup O_n$ be an ascending union of open subgroups.
    Let $\mu$ be a compact IRS of $G$.
    Assume that for every $n$ with $\mu(\Sub(O_n))>0$,
    the IRS of $O_n$ defined via $\mu_n(A) := \mu(A)/\mu(\Sub(O_n))$, for $A \subset \Sub(O_n)$, is contained in $\calE(O_n)$.
    Then $\ncm \le \calE(G)$.
\end{proposition}

\begin{proof}
    Let $\mu$ be a compact IRS of $G$.
    Assume without loss of generality that $\mu(\Sub(O_0))>0$.
    Consider the compact IRS $\mu_n$ of $O_n$ defined by $\mu_n(A) := \mu(A)/\mu(\Sub(O_n))$ for $A\subseteq \Sub(O_n)$. 
    
    Since the subgroups $O_n$ are open and exhaust $G$, for every compact subgroup $K \le G$ there exists $n_K\in\bbN$ with $K \le O_{n_K}$. In other words, $\calC(G) = \bigcup_n \calC(O_n)$, where $\calC(H)$ denotes the set of compact subgroups of a group $H$.
    Since $\mu (\calC(G))=1$, this implies that for every $\epsilon > 0$ there exists $k_\epsilon$ with $\mu(\Sub(O_{k_\epsilon})) > 1-\epsilon$.

    Let $n\geq k_\epsilon$. 
    Observe that for any $A\subseteq \Sub(O_{k_\epsilon})$, 
    \begin{align*}
        \mu_{k_\epsilon}(A)&=  \frac{\mu(A)}{\mu(\Sub(O_{k_\epsilon}))} \cdot \frac{\frac{1}{\mu(\Sub(O_n))}}{\frac{1}{\mu(\Sub(O_n))}} 
        = \frac{\mu_n(A)}{\mu_n(\Sub(O_{k_\epsilon}))}.
    \end{align*}    
and in particular,
    $$\mu_{k_\epsilon}(\Sub(\calE(O_n) \cap O_{k_\epsilon})) = \frac{\mu_n(\Sub(\calE(O_n)) \cap \Sub(O_{k_\epsilon}))}{\mu_n(\Sub(O_{k_\epsilon}))}.$$
By assumption, $O_n$ satisfies the weak conclusion, and so $\mu_n(\Sub(\calE(O_n)))=1$, it follows that the right hand side gives $1$.
    
    This implies that $\mu_{k_\epsilon}(\bigcap_{n \geq {k_\epsilon}} \Sub(\calE(O_n) \cap O_{k_\epsilon})) = 1$ and in particular $$\mu\Big(\bigcap_{n \geq {k_\epsilon}} \Sub(\calE(O_n))\Big) > 1-\epsilon.$$
    Finally, this shows that 
    \[
    1
    = \mu\Big(\bigcup_{{k} \in \bbN} \bigcap_{n \geq {k}} (\Sub(\calE(O_n)))\Big)
    = \mu\Big(\Sub\Big(\underbrace{\bigcup_{{k} \in \bbN} \bigcap_{n \geq {k}} (\calE(O_n))}_{=\calE(G)\text{ by Lemma \ref{lem:radicals and unions}}}\Big)\Big) 
    =\mu(\Sub(\calE(G))).
    \] 
\end{proof}
We obtain:
\begin{corollary}\label{cor:reduction of weak conc to cg}
    Let $G$ be an lcsc group such that every compactly generated subgroup of $G$ has the property that every compact IRS is contained in the locally elliptic radical. Then every compact IRS of $G$ is contained in the locally elliptic radical.
\end{corollary}

We now get that Theorem~\ref{thm:p-adic}\ref{general-p-adic} is an immediate consequence of Theorem \ref{thm:p-adic}\ref{cg-or-alg}.

\bibliographystyle{alpha}
\bibliography{references}

@article {Intersectional,
    AUTHOR = {Hartman, Yair and Yadin, Ariel},
     TITLE = {Furstenberg entropy of intersectional invariant random
              subgroups},
   JOURNAL = {Compos. Math.},
  FJOURNAL = {Compositio Mathematica},
    VOLUME = {154},
      YEAR = {2018},
    NUMBER = {10},
     PAGES = {2239--2265},
      ISSN = {0010-437X,1570-5846},
   MRCLASS = {37A40 (37A15 37A50 60B15 60G10)},
 MRREVIEWER = {Michael\ Stolz},
       DOI = {10.1112/s0010437x18007261},
       URL = {https://doi-org.ezproxy.weizmann.ac.il/10.1112/s0010437x18007261},
}

@article {7s,
    AUTHOR = {Abert, Miklos and Bergeron, Nicolas and Biringer, Ian and
              Gelander, Tsachik and Nikolov, Nikolay and Raimbault, Jean and
              Samet, Iddo},
     TITLE = {On the growth of {$L^2$}-invariants for sequences of lattices
              in {L}ie groups},
   JOURNAL = {Ann. of Math. (2)},
  FJOURNAL = {Annals of Mathematics. Second Series},
    VOLUME = {185},
      YEAR = {2017},
    NUMBER = {3},
     PAGES = {711--790},
      ISSN = {0003-486X,1939-8980},
   MRCLASS = {22E40 (22D40 22E30 58J35)},
 MRREVIEWER = {Lifan\ Guan},
       DOI = {10.4007/annals.2017.185.3.1},
       URL = {https://doi-org.ezproxy.weizmann.ac.il/10.4007/annals.2017.185.3.1},
}

@article {MR3624401,
    AUTHOR = {Biringer, Ian and Tamuz, Omer},
     TITLE = {Unimodularity of invariant random subgroups},
   JOURNAL = {Trans. Amer. Math. Soc.},
  FJOURNAL = {Transactions of the American Mathematical Society},
    VOLUME = {369},
      YEAR = {2017},
    NUMBER = {6},
     PAGES = {4043--4061},
      ISSN = {0002-9947,1088-6850},
   MRCLASS = {28C10 (22F10 37A20)},
       DOI = {10.1090/tran/6755},
       URL = {https://doi-org.ezproxy.weizmann.ac.il/10.1090/tran/6755},
}

@book {Glasner2003,
    AUTHOR = {Glasner, Eli},
     TITLE = {Ergodic theory via joinings},
    SERIES = {Mathematical Surveys and Monographs},
    VOLUME = {101},
 PUBLISHER = {American Mathematical Society, Providence, RI},
      YEAR = {2003},
     PAGES = {xii+384},
      ISBN = {0-8218-3372-3},
   MRCLASS = {37A15 (28Dxx 37A25 37A35 37A45 37B99 54H20)},
       DOI = {10.1090/surv/101},
       URL = {https://doi.org/10.1090/surv/101},
}

@inproceedings{Gelander2018,
series={London Mathematical Society Lecture Note Series},
title={A lecture on invariant random subgroups},
booktitle={New Directions in Locally Compact Groups},
publisher={Cambridge University Press},
author={Gelander, Tsachik},
editor={Caprace, Pierre-Emmanuel and Monod, NicolasEditors},
year={2018},
pages={186–204},
collection={London Mathematical Society Lecture Note Series}
}

@article{bader2016amenable,
  title={Amenable invariant random subgroups},
  author={Bader, Uri and Duchesne, Bruno and L{\'e}cureux, Jean and Wesolek, Phillip},
  journal = {Isr. J. Math.},
 fjournal={Israel Journal of Mathematics},
  volume={213},
  number={1},
  pages={399--422},
  year={2016},
  publisher={Springer}
}

@article{caprace2015amenable,
  title={Amenable hyperbolic groups},
  author={Caprace, Pierre-Emmanuel and De Cornulier, Yves and Monod, Nicolas and Tessera, Romain},
  journal={J. Eur. Math. Soc.},
  fjournal={Journal of the European Mathematical Society},
  volume={17},
  number={11},
  pages={2903--2947},
  year={2015}
}

@article {Cornulier_commability,
    AUTHOR = {Cornulier, Yves},
     TITLE = {Commability and focal locally compact groups},
   JOURNAL = {Indiana Univ. Math. J.},
  FJOURNAL = {Indiana University Mathematics Journal},
    VOLUME = {64},
      YEAR = {2015},
    NUMBER = {1},
     PAGES = {115--150},
      ISSN = {0022-2518},
   MRCLASS = {22D05 (20F67 57S20)},
MRREVIEWER = {Helge Gl\"{o}ckner},
       DOI = {10.1512/iumj.2015.64.5441},
       URL = {https://doi.org/10.1512/iumj.2015.64.5441},
}

@article{wang1971compactness,
  title={Compactness properties of topological groups},
  author={Wang, S. P.},
  journal={Trans. Amer. Math. Soc.},
  fjournal={Transactions of the American Mathematical Society},
  volume={154},
  pages={301--314},
  year={1971}
}

@article{trofimov1992action,
  title={On the action of a group on a graph},
  author={Trofimov, V. I.},
  journal={Acta Appl. Math.},
  fjournal={Acta Applicandae Mathematica},
  volume={29},
  number={1-2},
  pages={161--170},
  year={1992},
  publisher={Springer}
}

@article{HofmannWillis2015,
  title={Continuity characterizing totally disconnected locally compact groups},
  author={Hofmann, Karl H. and Willis, George A.},
  journal={J. Lie Theory},
  volume={25},
  number={1},
  pages={1--7},
  year={2015}
}

@article{Raja2025,
title={BOUNDED CONJUGACY CLASSES AND CONJUGACY CLASSES SUPPORTING INVARIANT MEASURES AND AUTOMORPHISMS},
volume={112}, DOI={10.1017/S0004972724001266},
number={3},
journal={Bull. Aust. Math. Soc.},
fjournal={Bulletin of the Australian Mathematical Society},
author={Raja, C. R. E.},
year={2025},
pages={509–514}
}

@article {Willis1994,
    AUTHOR = {Willis, George},
     TITLE = {The structure of totally disconnected, locally compact groups},
   JOURNAL = {Math. Ann.},
  FJOURNAL = {Mathematische Annalen},
    VOLUME = {300},
      YEAR = {1994},
    NUMBER = {2},
     PAGES = {341--363},
      ISSN = {0025-5831,1432-1807},
   MRCLASS = {22D05},
MRREVIEWER = {K.\ H.\ Hofmann},
       DOI = {10.1007/BF01450491},
       URL = {https://doi.org/10.1007/BF01450491},
}

@article{Caprace_Goffer_Lederle_Tsankov_2025, 
title={On compact uniformly recurrent subgroups}, 
DOI={10.1017/S0013091525100783}, 
journal={Proc. Edinb. Math. Soc.},
fjournal={Proceedings of the Edinburgh Mathematical Society}, 
author={Caprace, Pierre-Emmanuel and Goffer, Gil and Lederle, Waltraud and Tsankov, Todor}, 
year={2025}, 
pages={1–11}}

@article{Gas,
    title = "Gasch{\"u}tz {Lemma} for compact groups",
    author = "Tal Cohen and Tsachik Gelander",
    year = "2018",
    doi = "10.1016/j.jalgebra.2017.11.043",
    volume = "498",
    pages = "254--262",
    journal={J. {Algebra}},
    fjournal = "Journal of Algebra",
    issn = "0021-8693",
    publisher = "Academic Press Inc.",
}

@article{usakov1970classes,
  title={CLASSES OF CONJUGATE SUBGROUPS IN TOPOLOGICAL GROUPS},
  author={U\v{s}akov, V. I.},
  journal={Soviet Mathematics},
  volume={11},
  number={1},
  pages={48},
  year={1970},
  publisher={American Mathematical Society}
}

@article {PropertyM,
    AUTHOR = {Bader, Uri and Caprace, Pierre-Emmanuel and Gelander, Tsachik
              and Mozes, Shahar},
     TITLE = {Lattices in amenable groups},
   JOURNAL = {Fund. Math.},
  FJOURNAL = {Fundamenta Mathematicae},
    VOLUME = {246},
      YEAR = {2019},
    NUMBER = {3},
     PAGES = {217--255},
      ISSN = {0016-2736,1730-6329},
   MRCLASS = {22E40 (22D05 22F10 22F30)},
  MRREVIEWER = {Huibin\ Chen},
       DOI = {10.4064/fm572-9-2018},
       URL = {https://doi.org/10.4064/fm572-9-2018},
}

@article{stuck1994stabilizers,
  title={Stabilizers for ergodic actions of higher rank semisimple groups},
  author={Stuck, Garrett and Zimmer, Robert J},
  journal={Ann. of Math. (2)},
  fjournal={Annals of Mathematics},
  pages={723--747},
  year={1994},
  publisher={JSTOR}
}

@article{bowen2015invariant,
  title={Invariant random subgroups of the free group},
  author={Bowen, Lewis},
  journal={Groups Geom. Dyn.},
  fjournal={Groups, Geometry, and Dynamics},
  volume={9},
  number={3},
  pages={891--916},
  year={2015}
}

@book{Palais1960,
  author       = {Palais, Richard S.},
  title        = {The Classification of {$G$}-Spaces},
  series       = {Memoirs of the AMS},
  number       = {36},
  publisher    = {American Mathematical Society},
  address      = {Providence, RI},
  year         = {1960}
}

@article {FinitCentralMeasures,
    AUTHOR = {Sit, Kwan-Yuk Claire},
     TITLE = {On bounded elements of linear algebraic groups},
   JOURNAL = {Trans. Amer. Math. Soc.},
  FJOURNAL = {Transactions of the American Mathematical Society},
    VOLUME = {209},
      YEAR = {1975},
     PAGES = {185--198},
      ISSN = {0002-9947,1088-6850},
   MRCLASS = {22E20 (20G25)},
MRREVIEWER = {Linda\ Preiss\ Rothschild},
       DOI = {10.2307/1997378},
       URL = {https://doi-org.ezproxy.weizmann.ac.il/10.2307/1997378},
}

@incollection{MAT,
 author = {Gl{\"o}ckner, Helge},
 title = {Endomorphisms of {Lie} groups over local fields},
 booktitle = {2016 {MATRIX} annals},
 isbn = {978-3-319-72298-6; 978-3-319-72299-3},
 pages = {101--165},
 year = {2018},
 publisher = {Cham: Springer},
 doi = {10.1007/978-3-319-72299-3_6},
}

@article{UNI,
 author = {Gl{\"o}ckner, Helge and Willis, George A.},
 title = {Uniscalar {{\(p\)}}-adic {Lie} groups},
 fjournal = {Forum Mathematicum},
 journal = {Forum Math.},
 issn = {0933-7741},
 volume = {13},
 number = {3},
 pages = {413--421},
 year = {2001},
 doi = {10.1515/form.2001.015},
 url = {hdl.handle.net/1959.13/933251},
}

@book{SER,
 author = {Serre, Jean-Pierre},
 title = {Lie algebras and {Lie} groups},
 edition = {2nd ed.},
 fseries = {Lecture Notes in Mathematics},
 series = {Lect. Notes Math.},
 issn = {0075-8434},
 volume = {1500},
 isbn = {3-540-55008-9},
 year = {1992},
 publisher = {Berlin etc.: Springer-Verlag},
 doi = {10.1007/978-3-540-70634-2},
 zbMATH = {52847},
 Zbl = {0742.17008}
}

@article{PAR,
 author = {Parreau, Anne},
 title = {Elliptic subgroups of linear groups over a valued field},
 fjournal = {Journal of Lie Theory},
 journal = {J. Lie Theory},
 issn = {0949-5932},
 volume = {13},
 number = {1},
 pages = {271--278},
 year = {2003},
 language = {French},
 url = {https://eudml.org/doc/122882},
 zbMATH = {1873030},
 Zbl = {1021.20035}
}

@book{BOU,
 author = {Bourbaki, Nicolas},
 title = {{Lie} groups and {Lie} algebras. {Chapters} 1-3},
 edition = {2nd printing},
 isbn = {3-540-50218-1},
 year = {1989},
 publisher = {Berlin etc.: Springer-Verlag},
}

@article{Wan,
 author = {Wang, John S. P.},
 title = {The {Mautner} phenomenon for $p$-adic {Lie} groups},
 fjournal = {Mathematische Zeitschrift},
 journal = {Math. Z.},
 issn = {0025-5874},
 volume = {185},
 pages = {403--412},
 year = {1984},
 doi = {10.1007/BF01215048},
 url = {https://eudml.org/doc/173410},
 zbMATH = {3857404},
 Zbl = {0539.22015}
}

@article{Willis2023,
  title={Groups with flat-rank greater than $1$},
  author={Willis, George A},
  journal={J. Lie Theory},
  fjournal={Journal of Lie Theory},
  volume={33},
  number={1},
  pages={433--452},
  year={2023}
}

@article {AGV,
    AUTHOR = {Ab\'ert, Mikl\'os and Glasner, Yair and Vir\'ag, B\'alint},
     TITLE = {Kesten's theorem for invariant random subgroups},
   JOURNAL = {Duke Math. J.},
  FJOURNAL = {Duke Mathematical Journal},
    VOLUME = {163},
      YEAR = {2014},
    NUMBER = {3},
     PAGES = {465--488},
      ISSN = {0012-7094,1547-7398},
   MRCLASS = {20F69 (05C25 05C81 35P20 53C24)},
  MRNUMBER = {3165420},
MRREVIEWER = {B.\ Sury},
       DOI = {10.1215/00127094-2410064},
       URL = {https://doi-org.ezproxy.weizmann.ac.il/10.1215/00127094-2410064},
}

\end{document}